\documentclass[leqno]{amsart}
\usepackage{amsmath}
\usepackage{amssymb}
\usepackage{amsthm}
\usepackage{enumerate}
\usepackage[mathscr]{eucal}
\usepackage{graphicx}
\theoremstyle{plain}
\newtheorem{theorem}{Theorem}[section]

\newtheorem{lemma}{Lemma}[section]

\theoremstyle{definition}
\newtheorem{definition}{Definition}[section]
\newtheorem{remark}{Remark}[section]

\usepackage[pagewise]{lineno}
\begin{document}

\title[Norm attainment and semi-inner-products in normed spaces]{On the norm attainment set of a bounded linear operator and semi-inner-products in normed spaces}
\author[Debmalya Sain]{Debmalya Sain}

\newcommand{\acr}{\newline\indent}

\address{\llap{\,}Department of Mathematics\acr
                              Indian Institute of Science\acr
															Bengaluru\acr
															Karnataka 560012\acr
                              INDIA}
\email{saindebmalya@gmail.com}

\thanks{The research of the author is sponsored by Dr. D. S. Kothari Postdoctoral Fellowship under the mentorship of Professor Gadadhar Misra, to whom the author is immensely indebted. The author feels elated to acknowledge the colossal positive contribution of his alma mater, Ramakrishna Mission Vidyapith, Purulia, in every sphere of his life! He is extremely grateful to Professor Kallol Paul for a detailed first reading and valuable comments.} 

\subjclass[2010]{Primary 47A05, Secondary 46B20.}
\keywords{normed space; linear operator; norm attainment; semi-inner-product.}

\begin{abstract}
We obtain a complete characterization of the norm attainment set of a bounded linear functional on a normed space, in terms of a semi-inner-product defined on the space. Motivated by this result, we further apply the concept of semi-inner-products to obtain a complete characterization of the norm attainment set of a bounded linear operator between any two real normed spaces. In particular, this answers an open question raised recently in [Sain, D.,  \textit{On the norm attainment set of a bounded linear operator}, J. Math. Anal. Appl., \textbf{457} (2018), 67-76]. Our results illustrate the applicability of semi-inner-products towards a better understanding of the geometry of normed spaces.   
\end{abstract}

\maketitle

\section{Introduction.}

Norm of a vector is of fundamental importance in operator theory and functional analysis. The concept of norm of a bounded linear operator (or, functional) gives rise to the following two natural queries. First, whether the norm of the concerned operator is attained. And, if the answer to this question is in the affirmative, the second question is how to identify the points at which the norm is attained. While the first question has been explored and answered in great detail, the second question, surprisingly enough, had laid dormant for a long time. The primary purpose of this short note is to answer this question. As a matter of fact, we obtain a complete characterization of the norm attainment set of a bounded linear functional on any normed space. We further extend our study to completely characterize the norm attainment set of a bounded linear operator between real normed spaces. \\
Another major highlight of the present work is that through our study, we also illustrate the applicability of the concept of semi-inner-products in the study of geometry of normed spaces. In recent times, the structure and properties of the norm attainment set of a bounded linear operator between normed spaces has received wide attention \cite{S,Sa}. The present author obtained an interesting necessary condition for operator norm attainment in \cite{Sa}, for real normed spaces. In the same paper, complete characterization of the norm attainment set of a bounded linear operator between normed spaces was identified as an important open question in the geometry of Banach spaces. In this paper, we obtain a solution to this problem by applying the concept of semi-inner-products. Without further ado, let us establish the relevant notations and terminologies to be used throughout the paper.\\

Let $ \mathbb{X},~\mathbb{Y} $ be normed spaces of dimension strictly greater than $ 1, $ over the field $ \mathbb{K}, $ real or complex. Without mentioning any further, we assume that $ \mathbb{X},~\mathbb{Y} $ are inherently linear spaces. Let $ B_{\mathbb{X}} = \{ x \in \mathbb{X} : \| x \| \leq 1\} $ and $ S_{\mathbb{X}} = \{ x \in \mathbb{X} : \| x \| =1 \} $ denote the unit ball and the unit sphere of $ \mathbb{X} $ respectively. Let $ \mathbb{X}^{*} $ be the dual space of $ \mathbb{X}. $ For any $ f \in \mathbb{X}^{*}, $ let $ \textit{N}(f) $ denote the null space of $ f. $ Given any $ x \in S_{\mathbb{X}}, $ a functional $ \psi_{x} \in S_{\mathbb{X}^{*}} $ is said to be a supporting functional at $ x $ if $ \| \psi_{x} \| = 1 =  \psi_{x} (x). $ It is worth noticing that the Hahn-Banach theorem guarantees the existence of at least one supporting functional at each point of $ S_{\mathbb{X}}. $ Given any $ \lambda x \in \mathbb{X}, $ where $ \lambda \in \mathbb{K} $ and $ x \in S_{\mathbb{X}}, $ we define, $ \psi_{\lambda x} = \overline{\lambda} \psi_{x}. $ We say that $ \mathbb{X} $ is smooth if given any $ x \in S_{\mathbb{X}}, $ there exists a unique supporting functional at $ x. $ In other words, if $ \mathbb{X} $ is smooth then given any $ x \in S_{\mathbb{X}}, $ there exists a unique supporting functional $ \psi_{x} $ at $ x. $ Given $ T \in \mathbb{L}(\mathbb{X},\mathbb{Y}), $ let $ M_{T} $ denote the norm attainment set of $ T, $  i.e., $ M_{T} = \{ x \in S_{\mathbb{X}} : \|Tx\| = \| T \|\}. $ We next mention the concept of semi-inner-products in normed spaces, which is integral to the theme of this paper. 

\begin{definition}
Let $ \mathbb{X} $ be a normed space. A function $ [ ~,~ ] : \mathbb{X} \times \mathbb{X} \longrightarrow \mathbb{K}(=\mathbb{R},~\mathbb{C}) $ is a semi-inner-product (s.i.p.) if and only if for any $ \alpha,~\beta \in \mathbb{K} $ and for any $ x,~y,~z \in \mathbb{X}, $ it satisfies the following:\\
$ (a) $ $ [\alpha x + \beta y, z] = \alpha [x,z] + \beta [y,z], $\\
$ (b) $ $ [x,x] > 0, $ whenever $ x \neq 0, $\\
$ (c) $ $ |[x,y]|^{2} \leq [x,x] [y,y], $\\
$ (d) $ $ [x,\alpha y] = \overline{\alpha} [x,y]. $
\end{definition} 

Giles proved in \cite{G} that every normed space $ (\mathbb{X},\|\|) $ can be represented as an s.i.p. space $ (\mathbb{X},[~,~]) $ such that for all $ x \in \mathbb{X}, $ we have, $ [x,x] = \| x \|^{2}. $ In general, there can be many compatible s.i.p. corresponding to a given norm. Whenever we speak of a s.i.p. $ [~,~] $ in context of a normed space $ \mathbb{X} $, we implicitly assume that  $ [~,~] $ is compatible with the norm, i.e., for all $ x \in \mathbb{X}, $ we have, $ [x,x] = \| x \|^{2}. $ Lumer stated in \cite{L} that there exists a unique s.i.p. on a normed space if and only if the space is smooth. If $ \mathbb{X} $ is smooth, then the unique s.i.p. on $ \mathbb{X} $ can easily be seen to be the following: $ [x,y] = \psi_{y}(x) $ for all $ x,~y \in \mathbb{X}. $\\  

We make use of the notion of Birkhoff-James orthogonality \cite{B,J} in normed spaces to relate s.i.p. with norm attainment of a bounded linear operator. An element $ x \in \mathbb{X} $ is said to be orthogonal to another element $ y \in \mathbb{X} $ in the sense of Birkhoff-James, written as $ x \perp_{B} y, $ if $ \| x + \lambda y \| \geq \| x \| $ for all scalars $ \lambda. $ While studying Birkhoff-James orthogonality of bounded linear operators between $\textit{real}$ normed spaces, the author introduced the following notations in \cite{S}: For any two elements $ x, y $ in $ \mathbb{X}, $ let us say that $ y \in x^{+} $ if $ \| x + \lambda y \| \geq \| x \| $ for all $ \lambda \geq 0. $ Accordingly, we say that $ y \in x^{-} $ if $ \| x + \lambda y \| \geq \| x \| $ for all $ \lambda \leq 0. $ Let  $ x^{\perp} = \{ y \in \mathbb{X} : x \perp_{B} y. \} $\\

Using the concept of s.i.p., we first obtain a complete characterization of the norm attainment set of a bounded linear functional on any normed space. This evidently illustrates the strength of s.i.p. type arguments in studying the geometry of normed spaces. Motivated by this result, we next obtain a complete characterization of the norm attainment set of a bounded linear operator between real normed spaces. In \cite{Sa}, the author posed the following open question in the context of real normed spaces:\\

\emph{Let $ T $ be a bounded linear operator defined on a Banach space $ \mathbb{X}. $ Find a necessary and sufficient condition for $ x \in S_{\mathbb{X}} $ to be such that $ x \in M_{T}. $} \\

Our result gives a complete solution to this problem by applying the concept of s.i.p. in normed spaces. We would also like to mention that our answer is not restricted by either of the conditions of completeness and finite-dimensionality assumed in the above mentioned problem. Furthermore, we do not require the domain space and the range space to be identical. As an immediate application of our approach, it is possible to characterize the isometries between any two real normed spaces.      
   
\section{ Main Results.}
As promised in the introduction, we begin with a complete characterization of the norm attainment set of a bounded linear functional on a normed space.

\begin{theorem}
Let $ \mathbb{X} $ be a normed space and $ f \in \mathbb{X}^{*}. $ Let $ z \in S_{\mathbb{X}}. $ Then $ z \in M_{f} $ if and only if there exists a s.i.p. $ [~,~] $ on $ \mathbb{X} $ such that $ f(x) = [x,\|f\| z] $ for any $ x \in \mathbb{X}. $ 
\end{theorem}

\begin{proof}
Let us first prove the easier ``if" part of the theorem. Suppose there exists a s.i.p. $ [~,~] $ on $ \mathbb{X} $ such that $ f(x) = [x,\|f\| z] $ for any $ x \in \mathbb{X}. $ Putting $ x = z,  $ we obtain, 
\[ f(z) = [z,\|f\| z] = \| f \| [z,z] = \| f \| \| z \|^{2} = \| f \|. \] 
This proves that $ z \in M_{f} $ and completes the proof of the theorem in one direction.\\

Let us now prove the comparatively trickier ``only if" part of the theorem. We first observe that if $ f = 0 $ then the theorem holds trivially. Let $ f $ be nonzero and let $ z \in M_{f}. $ We employ the same arguments as given by Giles in \cite{G}, with a minor modification, to obtain a s.i.p. on $ \mathbb{X}. $ By virtue of the Hahn-Banach theorem, given any $ x \in S_{\mathbb{X}}, $ there exists at least one bounded linear functional $ g_{x} \in \mathbb{X}^{*} $ such that $ g_{x}(x) = \| g_x \| = 1. $ Following Giles \cite{G}, we choose exactly one such $ g_x. $ For $ \lambda x \in \mathbb{X}, $ where $ \lambda \in \mathbb{K} $ and $ x \in S_{\mathbb{X}}, $ we choose $ g_{\lambda x} \in \mathbb{X}^{*} $ such that $ g_{\lambda x} = \overline{\lambda} g_{x}. $ Our only modification is the following:\\

It is easy to see that since $ f $ is nonzero, $ \frac{f}{\|f\|} $ is a candidate for $ g_{z}. $ We choose $ g_{z} = \frac{f}{\|f\|}. $\\

We now define a s.i.p. $ [~,~] $ on $ \mathbb{X} $ by $ [x,y] = g_{y}(x) $ for all $ x,~y \in \mathbb{X}. $ It can be easily verified that $ [~,~] $ satisfies all the properties $ (a) - (d) $ of s.i.p., so that $ [~,~] $ is indeed a s.i.p. on $ \mathbb{X}. $ We next observe that given any $ w \in \textit{N}(f), $ $ [w,z] = g_z(w) = \frac{1}{\|f\|} f(w) = 0. $\\
Given any $ x \in \mathbb{X}, $ consider the element $ u = f(z) x - f(x) z.  $ Clearly, $ u \in \textit{N}(f). $ Therefore, we have,
\[ 0 = [u,z] = f(z) [x,z] - f(x) [z,z] = \|f\| [x,z] - f(x) \|z\|^{2} = \|f\| [x,z] - f(x). \]

In other words, we have, $ f(x) = \|f\| [x,z] = [x,\|f\|z]. $ Since this is true for any $ x \in \mathbb{X}, $ this completes the proof of the theorem in the other direction and establishes it completely.

\end{proof}

In view of Theorem $ 2.1, $ it is natural to inquire about an analogous characterization of the norm attainment set of a bounded linear operator between normed spaces, instead of bounded linear functionals. In the remaining part of this paper, we illustrate that it is possible to once again apply s.i.p. type arguments to answer this question for real normed spaces. Our initial objective is to consider the case where both the domain space and the range space are smooth.

\begin{theorem}
Let $ \mathbb{X},~\mathbb{Y} $ be smooth real normed spaces and $ T \in \mathbb{L}(\mathbb{X},\mathbb{Y}). $ Let $ z \in S_{\mathbb{X}}. $ Then $ z \in M_{T} $ if and only if $ [Tx,Tz] = \|T\|^{2} [x,z] $ for any $ x \in \mathbb{X}. $ 
\end{theorem}    

\begin{proof}
The ``if" part of the theorem follows easily as before, by putting $ x=z. $ Let us prove the nontrivial ``only if" part of the theorem. Clearly, the result holds if $ T $ is the zero operator. Let us assume that $ T $ is nonzero. Let $ z \in S_{\mathbb{X}} $ be such that $ z \in M_{T}. $ Since $ \mathbb{X} $ and $ \mathbb{Y} $ are smooth spaces, it follows from Theorem $ 2.3 $ of \cite{Sa} that $ T(z^{\perp}) \subseteq (Tz)^{\perp}. $ We also note that there exist a unique s.i.p. on either of $ \mathbb{X} $ and $ \mathbb{Y}. $ We claim that $ \psi_{z} (h) = \psi_{Tz}(Th) = 0, $ for any $ h \in z^{\perp}, $ where $ \psi_{z} $ is the support functional at $ z \in S_{\mathbb{X}}, $ $ \psi_{\frac{Tz}{\|Tz\|}} $ is the support functional at $ \frac{Tz}{\|Tz\|} \in S_{\mathbb{Y}} $ and $ \psi_{Tz} = \| Tz \| \psi_{\frac{Tz}{\|Tz\|}}. $ We note that since $ T $ is nonzero and $ z \in M_{T}, $ we must have, $ Tz \neq 0. $ To prove our claim, we argue in the following way:\\

Clearly, $ \psi_{z} $ attains norm at $ z. $ Therefore, it follows that $ z \perp_{B} \textit{N}(\psi_{z}). $ As $ \mathbb{X} $ is smooth, there exists a unique hyperplane $ H $ of codimension $ 1 $ such that $ z \perp_{B} H. $ It is trivial to observe that $ \textit{N}(\psi_{z}) $ is a hyperplane of codimension $ 1. $ Therefore, we must have, $ H = \textit{N}(\psi_{z}). $ Since $ z \perp_{B} h, $ it necessarily follows that $ h \in \textit{N}(\psi_{z}), $ i.e., $ \psi_{z}(h) = 0. $ Similarly, using the smoothness of $ \mathbb{Y} $ and the fact that $ Tz \perp_{B} Th, $ we further have, $ Th \in \textit{N}(\psi_{Tz}), $ i.e., $ \psi_{Tz} (Th) = 0. $ This completes the proof of our claim.\\

We observe that given any $ x \in \mathbb{X}, $ there is a unique representation $ x = \alpha z + h, $ where $ \alpha \in \mathbb{R} $ and $ h \in H. $ Now, we have,
\[ [Tx,Tz] = [\alpha Tz + Th, Tz] = \alpha [Tz,Tz] + [Th,Tz] = \alpha \|Tz\|^{2} + \psi_{Tz} (Th) = \alpha \|T\|^{2}. \] 

On the other hand, we also have,
\[ [x,z] = [\alpha z + h, z] = \alpha [z,z] + [h,z] = \alpha \|z\|^{2} + \psi_{z}(h) = \alpha. \] 

Therefore, we have, $ [Tx,Tz] = \|T\|^{2} [x,z]. $ Since $ x \in \mathbb{X} $ was chosen arbitrarily, this establishes the theorem.
\end{proof}

\begin{remark}
It is worthwhile to observe that the crux of the proof of the above theorem is practically dependent on Theorem $ 2.3 $ of \cite{Sa}, that states that if a bounded linear operator between smooth real normed spaces attains norm at a particular point of the unit sphere then the operator preserves Birkhoff-James orthogonality at that point. This explains why we choose to work initially with the additional assumption of smoothness.  
\end{remark}

Next, we would like to extend Theorem $ 2.2 $ to any pair of real normed spaces. As explained in Remark $ 2.1, $ first we need the following two lemmas, in order to replace the relevant part of Theorem $ 2.3 $ of \cite{Sa} in our arguments.

\begin{lemma}
Let $ \mathbb{X} $ be a normed space and let $ \textit{A},~\textit{B} $ be two nonempty disjoint open subsets of $ \mathbb{X}. $ Then there does not exist a path in $ \mathbb{X} $ that starts at a point in $ \textit{A}, $ ends at a point in $ \textit{B} $ and lies entirely within $ \textit{A}~\bigcup~\textit{B}. $
\end{lemma}

\begin{proof}
We prove the result by the method of contradiction. If possible, suppose that there exists a path $ \{ f(t) : t \in [0,1] \} $ in $ \mathbb{X}, $ where $ f:[0,1] \longrightarrow \mathbb{X} $ is continuous, such that $ f(0) \in \textit{A}, $ $ f(1) \in \textit{B} $ and $ f(t) \in \textit{A}~\bigcup~\textit{B} $ for any $ t \in [0,1]. $ Let $ \textit{C} = \{ t \in [0,1] : f(t) \in \textit{A}. \} $ Clearly, $ \textit{C} $ is a nonempty subset of $ \mathbb{R}, $ which is bounded above. Let $ t_0 = sup~\textit{C} \in [0,1]. $ Let $ x_0 = f(t_0). $ Since $ f(0) \in \textit{A},~f(1) \in \textit{B} $ and $ \textit{A},~\textit{B} $ are open, it follows that $ t_0 \in (0,1). $ According to our assumptions, either $ x_0 \in \textit{A} $ or $ x_0 \in \textit{B}. $ We will obtain a contradiction in each case to complete the proof of the lemma.\\
Let us first assume that $ x_0 = f(t_0) \in \textit{A}. $ Since $ \textit{A} $ is open, there exists $ \delta_0 > 0 $ such that $ f(t_0 - \delta_0,t_0+\delta_0) \subset \textit{A}. $ However, this implies that $ t_0 + \frac{\delta_0}{2} \in \textit{C}, $ contradicting our assumption that $ t_0 = sup~\textit{C}. $\\
Let us now assume that $ x_0 = f(t_0) \in \textit{B}. $ Since $ \textit{B} $ is open, there exists $ \delta_0 > 0 $ such that $ f(t_0 - \delta_0,t_0+\delta_0) \subset \textit{B}. $ Since $ \textit{A} $ and $ \textit{B} $ are disjoint, $ (t_0 - \delta_0,t_0+\delta_0) \notin \textit{A}. $ Moreover, for any $ t \geq t_0, $ we must have, $ t \notin \textit{C}. $ Therefore it follows that $ t_0 = sup~\textit{C} \leq t_0 - \delta_0, $ a contradiction.\\
This completes the proof of the lemma.  
\end{proof}

\begin{lemma}
Let $ \mathbb{X},~\mathbb{Y} $ be normed spaces and $ T \in \mathbb{L}(\mathbb{X},\mathbb{Y}). $ Let $ z \in S_{\mathbb{X}} $ be such that $ z \in M_{T}. $ Then there exists $ y \in \mathbb{X} \setminus \{0\} $ such that $ z \perp_{B} y $ and $ Tz \perp_{B} Ty. $
\end{lemma}

\begin{proof}
Clearly, the result is true if $ T $ is the zero operator. Let us assume that $ T $ is nonzero. It follows from Proposition $ 2.1 $ of \cite{S} that $ \mathbb{X} =  (z^{+} \setminus z^{\perp}) \cup (z^{-} \setminus z^{\perp}) \cup (z^{\perp}).  $ Let us choose $ u \in (z^{+} \setminus z^{\perp}) $ and $ v \in (z^{-} \setminus z^{\perp}) $ such that the line segment joining $ u $ and $ v $ does not pass through origin. We note that such a choice of $ u $ and $ v $ is always possible. We further note that the first two statements of Theorem $ 2.3 $ of \cite{Sa} does not depend on the smoothness of the concerned spaces. Therefore, we have, $ Tu \in ((Tz)^{+} \setminus (Tz)^{\perp}) $ and $ Tv \in ((Tz)^{-} \setminus (Tz)^{\perp}). $ Therefore, applying the continuity of $ T, $ it is easy to see that we have the following:\\

$ \{ T((1-t)u+tv) : t \in [0,1] \} $ is a path in $ \mathbb{Y}, $ that begins at $ Tu \in ((Tz)^{+} \setminus (Tz)^{\perp}) $ and ends at $ Tv \in ((Tz)^{-} \setminus (Tz)^{\perp}). $\\

Since $ (Tz)^{+} \cap (Tz)^{-} = (Tz)^{\perp},  $ it follows that $ ((Tz)^{+} \setminus (Tz)^{\perp})~ \bigcap~ ((Tz)^{-} \setminus (Tz)^{\perp}) = \emptyset. $ It is also easy to observe that since $ T $ is nonzero, both $ \textit{A}=((Tz)^{+} \setminus (Tz)^{\perp}) $ and $ \textit{B}=((Tz)^{-} \setminus (Tz)^{\perp}) $ are nonempty open subsets of $ \mathbb{Y}. $ Therefore, by Lemma $ 2.1, $ there cannot exist a path that starts at a point in $ \textit{A}, $ ends at a point in $ \textit{B} $ and lies completely within $ \textit{A}~ \bigcup~ \textit{B}. $ In other words, there exists $ t_{0} \in (0,1) $ such that $ T((1-t_0)u+t_0v) \notin \textit{A}~ \bigcup~ \textit{B}. $ On the other hand, since $ T(\mathbb{X}) = \textit{A}~ \bigcup~ \textit{B}~ \bigcup~ (Tz)^{\perp}, $ it follows that $ T((1-t_0)u+t_0v) \in (Tz)^{\perp}. $   Once again, it follows from $ (i) $ and $ (ii) $ of Theorem $ 2.3 $ of \cite{Sa}, that $ (1-t_0)u+t_0 v \notin (z^{+} \setminus z^{\perp})~\bigcup~(z^{+} \setminus z^{\perp}). $ Since $ \mathbb{X} =  (z^{+} \setminus z^{\perp}) \cup (z^{-} \setminus z^{\perp}) \cup (z^{\perp}),  $ we must have, $ (1-t_0)u+t_0 v \in z^{\perp}. $ Let us choose $ y = (1-t_0)u+t_0 v. $ It follows from our choice of $ u $ and $ v $ that $ y \neq 0. $ Furthermore, we have already proved that $ z \perp_{B} y $ and $ Tz \perp_{B} Ty. $ This completes the proof of the lemma.

\end{proof} 

Now, the promised characterization of the norm attainment set of a bounded linear operator between any two real normed spaces:\\

\begin{theorem}
Let $ \mathbb{X},~\mathbb{Y} $ be real normed spaces and $ T \in \mathbb{L}(\mathbb{X},\mathbb{Y}). $ Let $ z \in S_{\mathbb{X}}. $ Then $ z \in M_{T} $ if and only if given any $ x \in \mathbb{X}, $ there exists two s.i.p. $ [~,~]_{\mathbb{X}} $ and $ [~,~]_{\mathbb{Y}} $ on $ \mathbb{X} $ and $ \mathbb{Y} $ respectively such that 
\[ [Tx,Tz]_{\mathbb{Y}} = \| T \|^{2} [x,z]_{\mathbb{X}}. \]
\end{theorem}

\begin{proof}
The ``if" part of the theorem follows trivially, as before. Let us prove the ``only if" part of the theorem. Without loss of generality, we may and do assume that $ T $ is nonzero. Let $ z \in M_{T}. $ Let $ x \in \mathbb{X} $ be arbitrary. If $ x=\lambda z, $ for some $ \lambda \in \mathbb{R}, $ then for any two s.i.p. $ [~,~]_{\mathbb{X}} $ and $ [~,~]_{\mathbb{Y}} $ on $ \mathbb{X} $ and $ \mathbb{Y} $ respectively, we have,
\[ [Tx,Tz]_{\mathbb{Y}} = [\lambda Tz,Tz]_{\mathbb{Y}} = \lambda \|Tz\|^2 = \|T\|^2 \lambda = \|T\|^2 [\lambda z,z]_{\mathbb{X}} = \|T\|^2 [x,z]_{\mathbb{X}}. \]

Let us assume that $ x,~z $ are linearly independent. Let $ \mathbb{Z} $ be the two-dimensional normed space spanned by $ \{x,z\}, $ with the induced norm. By Lemma $ 2.2, $ there exists $ y \in \mathbb{Z} \setminus \{0\} $ such that $ z \perp_{B} y $ and $ Tz \perp_{B} Ty. $ We would like to construct two s.i.p. $ [~,~]_{\mathbb{X}} $ and $ [~,~]_{\mathbb{Y}} $ on $ \mathbb{X} $ and $ \mathbb{Y} $ respectively such that the desired condition holds.\\

Clearly, $ x = \lambda_0 z + \mu_0 y, $ for some scalars $ \lambda_0,~\mu_0 \in \mathbb{R}. $ Since $ z \perp_{B} y, $ there exists a functional $ l_{z}:\mathbb{Z} \longrightarrow \mathbb{R} $ such that $ \| l_{z} \| = l_{z}(z) = \|z\| = 1 $ and $ l_{z}(y) = 0. $ By the Hahn-Banach theorem, $ l_{z} $ has a norm preserving extension $ \psi_{z}:\mathbb{X} \longrightarrow \mathbb{R}. $ Clearly, $ \psi_{z}(y) = 0. $ In order to define a s.i.p. on $ \mathbb{X}, $ we once again argue as Giles in \cite{G}, with a minor modification. By virtue of the Hahn-Banach theorem, given any $ w \in S_{\mathbb{X}}, $ there exists at least one bounded linear functional $ g_{w} \in \mathbb{X}^{*} $ such that $ g_{w}(w) = \| g_w \| = 1. $ Following Giles \cite{G}, we choose exactly one such $ g_w. $ For $ \lambda w \in \mathbb{X}, $ where $ \lambda \in \mathbb{R} $ and $ w \in S_{\mathbb{X}}, $ we choose $ g_{\lambda w} \in \mathbb{X}^{*} $ such that $ g_{\lambda w} = \lambda g_{w}. $ Our only modification is the following:\\
 
It is easy to see that $ \psi_{z} $ is a candidate for $ g_{z}. $ We choose $ g_{z} = \psi_{z}. $\\

We now define a s.i.p. $ [~,~]_{\mathbb{X}} $ on $ \mathbb{X} $ by $ [v,w] = g_{w}(v) $ for any $ v,~w \in \mathbb{X}. $ It is easy to check that $ [~,~]_{\mathbb{X}} $ is indeed a s.i.p. in $ \mathbb{X}. $ Moreover, we have,
\[ [x,z]_{\mathbb{X}} = [\lambda_0 z + \mu_0 y,z]_{\mathbb{X}} = \lambda_0 [z,z]_{\mathbb{X}} + \mu_0 [y,z]_{\mathbb{X}} = \lambda_0 \|z\|^2 = \lambda_0. \] 

On the other hand, since $ T $ is nonzero and $ z \in M_{T}, $ we must have, $ Tz $ is nonzero in $ \mathbb{Y}. $ Furthermore, we have already proved that $ Tz \perp_{B} Ty. $  Therefore, following similar method, we can define a s.i.p. $ [~,~]_{\mathbb{Y}} $ on $ \mathbb{Y} $ such that $ [Ty,Tz]_{\mathbb{Y}} = 0. $ Thus, we have,\\
\[ [Tx,Tz]_{\mathbb{Y}} = [\lambda_0 Tz + \mu_0 Ty,Tz]_{\mathbb{Y}} = \lambda_0 [Tz,Tz]_{\mathbb{Y}} + \mu_0 [Ty,Tz]_{\mathbb{Y}} = \lambda_0 \|T\|^2 = \| T \|^2 [x,z]_{\mathbb{X}}. \] 

Since $ x \in \mathbb{X} $ was chosen arbitrarily, this establishes the theorem.
\end{proof}

\begin{remark}
As mentioned in the introduction, Theorem $ 2.3 $ gives a complete answer to the open question posed in \cite{Sa}, regarding a necessary and sufficient condition for a bounded linear operator to attain norm at a given point of the unit sphere. Moreover, we do not require the additional assumptions of completeness, nor do we require the normed spaces to be finite-dimensional or identical.
\end{remark}

As an example of another powerful application of s.i.p. in the study of normed spaces, it is possible to characterize the set of all isometries between any two real normed spaces. Indeed, the following theorem is immediate:

\begin{theorem}
Let $ \mathbb{X},~\mathbb{Y} $ be real normed spaces and $ T \in \mathbb{L}(\mathbb{X},\mathbb{Y}). $ Then $ T $ is an isometry if and only if given any $ x \in \mathbb{X} $ and any $ z \in S_{\mathbb{X}}, $ there exists two s.i.p. $ [~,~]_{\mathbb{X}} $ and $ [~,~]_{\mathbb{Y}} $ on $ \mathbb{X} $ and $ \mathbb{Y} $ respectively such that 
\[ [Tx,Tz]_{\mathbb{Y}} = \| T \|^{2} [x,z]_{\mathbb{X}}. \]
\end{theorem}

In view of the extensive use of the concept of s.i.p. towards obtaining a complete characterization of the norm attainment set of a bounded linear operator between real normed spaces, it is perhaps appropriate that we end the present paper with the following remark:

\begin{remark}
As mentioned by both Lumer \cite{L} and Giles \cite{G}, the original motivation behind introducing s.i.p. in normed spaces was to imitate inner product space type arguments in the much broader setting of normed spaces. Our results in this paper further illustrate  the remarkable success of the concept of s.i.p. towards a better understanding of the geometry of normed spaces in general and operator norm attainment in particular.
\end{remark}

\bibliographystyle{amsplain}

\begin{thebibliography}{99}

\bibitem{B} Birkhoff, G.,  \textit{Orthogonality in linear metric spaces}, Duke Math. J., \textbf{1} (1935), 169-172.


\bibitem{G} Giles, J. R.,  \textit{Classes of semi-inner-product spaces}, Trans. Amer. Math. Soc., \textbf{129} (1967), 436-446.


\bibitem{J} James, R. C.,  \textit{Orthogonality and linear functionals in normed linear spaces}, Trans. Amer. Math. Soc., \textbf{61} (1947), 265-292.


\bibitem{L} Lumer, G.,  \textit{Semi-inner-product spaces}, Trans. Amer. Math. Soc., \textbf{100} (1961), 29-43.


\bibitem{S} Sain, D.,  \textit{Birkhoff-James orthogonality of linear operators on finite dimensional Banach spaces}, J. Math. Anal. Appl., \textbf{447} (2017), 860-866.



\bibitem{Sa} Sain, D.,  \textit{On the norm attainment set of a bounded linear operator}, J. Math. Anal. Appl., \textbf{457} (2018), 67-76.









  




\end{thebibliography}

\end{document}